\documentclass[reqno]{amsart}
\usepackage{amsmath}
\usepackage{amsfonts}
\usepackage{amssymb,enumerate}
\usepackage{amsthm}
\usepackage[all]{xy}
\usepackage{rotating}
\usepackage{hyperref}
\usepackage{color}

\theoremstyle{plain}
\newtheorem{lem}{Lemma}[section]
\newtheorem{cor}[lem]{Corollary}
\newtheorem{prop}[lem]{Proposition}
\newtheorem{thm}[lem]{Theorem}

\newtheorem{intthm}{Theorem}

\newtheorem*{mthm*}{Main Theorem}

\theoremstyle{definition}
\newtheorem{defn}[lem]{Definition}
\newtheorem{ex}[lem]{Example}

\newtheorem{disc}[lem]{Remark}

\newtheorem{fact}[lem]{Fact}
\newtheorem{para}[lem]{}

\newtheorem{convention}[lem]{Convention}
\newtheorem*{convention*}{Convention}

\newcommand{\A}{\mathcal{A}}

\newcommand{\pd}{\operatorname{pd}}

\newcommand{\id}{\operatorname{id}}

\newcommand{\depth}{\operatorname{depth}}

\newcommand{\edim}{\operatorname{edim}}

\newcommand{\HH}{\operatorname{H}}

\newcommand{\spec}{\operatorname{Spec}}

\newcommand{\Ker}{\operatorname{Ker}}

\newcommand{\ideal}[1]{\mathfrak{#1}}
\newcommand{\m}{\ideal{m}}

\newcommand{\p}{\ideal{p}}
\newcommand{\q}{\ideal{q}}
\newcommand{\fm}{\ideal{m}}

\newcommand{\xra}{\xrightarrow}
\newcommand{\xla}{\xleftarrow}

\newcommand{\res}{\xra{\simeq}}

\newcommand{\e}{\mathbf{e}}

\newcommand{\Ext}[4][R]{\operatorname{Ext}_{#1}^{#2}(#3,#4)}

\newcommand{\Hom}{\operatorname{Hom}}	
\newcommand{\Tor}[4][R]{\operatorname{Tor}^{#1}_{#2}(#3,#4)}

\def\Tor{\operatorname{Tor}}
\def\Ext{\operatorname{Ext}}

\def\m{\mathfrak{m}}

\def\p{\mathfrak{p}}
\def\mod{\operatorname{\mathsf{mod}}}
\def\syz{\mathrm{\Omega}}
\def\thick{\operatorname{\mathsf{thick}}}
\def\X{\mathcal{X}}
\def\Z{\mathcal{Z}}
\def\A{\mathcal{A}}
\def\T{\mathcal{T}}
\def\xx{\boldsymbol{x}}

\def\PD{\mathsf{PD}}
\def\sing{\operatorname{\mathsf{Sing}}}
\def\spec{\operatorname{\mathsf{Spec}}}
\def\speco{\operatorname{\mathsf{Spec^0}}}
\def\cm{\mathsf{CM}}
\def\nf{\mathsf{NF}}
\def\ipd{\mathsf{IPD}}
\def\rest{\operatorname{\mathsf{rest}}}
\def\ng{\operatorname{\mathsf{NonGor}}}

\def\k{\mathrm{K}}
\def\res{\operatorname{\mathsf{res}}}
\def\add{\operatorname{\mathsf{add}}}
\def\D{\operatorname{\mathsf{D}}}
\def\e{\mathrm{e}}
\def\edim{\operatorname{edim}}
\def\I{\mathrm{I_2}\!}
\def\sg{\operatorname{\mathsf{sg}}}
\def\b{\operatorname{\mathsf{b}}}

\def\Y{\mathcal{Y}}
\def\ssupp{\operatorname{\mathsf{ssupp}}}

\numberwithin{equation}{lem}

\begin{document}

\bibliographystyle{amsplain}

\title[Local rings with quasi-decomposable maximal ideal]{Local rings with quasi-decomposable\\ maximal ideal}

\author{Saeed Nasseh}

\address{Department of Mathematical Sciences,
Georgia Southern University,
Statesboro, Georgia 30460, U.S.A.}

\email{snasseh@georgiasouthern.edu}
\urladdr{https://cosm.georgiasouthern.edu/math/saeed.nasseh}

\author{Ryo Takahashi}
\address{Graduate School of Mathematics\\
Nagoya University\\
Furocho, Chikusaku, Nagoya, Aichi 464-8602, Japan}
\email{takahashi@math.nagoya-u.ac.jp}
\urladdr{http://www.math.nagoya-u.ac.jp/~takahashi/}

\thanks{Takahashi was partly supported by JSPS Grants-in-Aid
for Scientific Research 16K05098.}

%\dedicatory{}


\begin{abstract}
Let $(R,\m)$ be a commutative noetherian local ring.
In this paper, we prove that if $\m$ is decomposable, then for any finitely generated $R$-module $M$ of infinite projective dimension $\m$ is a direct summand of (a direct sum of) syzygies of $M$.
Applying this result to the case where $\m$ is quasi-decomposable, we obtain several classifications of subcategories, including a complete classification of the thick subcategories of the singularity category of $R$.
\end{abstract}

\maketitle

%tableofcontents

\section{Introduction}\label{sec161010a}
Let $R$ be a commutative noetherian local ring with maximal ideal $\m$.
First, we investigate the structure of syzygies of finitely generated $R$-modules in the case where $\m$ is decomposable as an $R$-module.
Our main result in this direction is Theorem \ref{1}, which includes the following remarkable statement.

\begin{intthm}\label{cor}
Let $R$ be a commutative noetherian local ring with decomposable maximal ideal $\m$.
Let $M$ be a finitely generated $R$-module with infinite projective dimension. Then $\m$ is a direct summand of $\syz^3M\oplus\syz^4M\oplus\syz^5M$.
\end{intthm}

\noindent
Here, $\syz^n(-)$ stands for the $n$-th syzygy in the minimal free resolution.
This result (Theorem \ref{1} strictly) recovers and refines main theorems of the first author and Sather-Wagstaff \cite{NSW} on the vanishing of Tor and Ext modules over fiber products.

Next, we apply the above theorem to classification problems of subcategories over a Cohen-Macaulay local ring with quasi-decomposable maximal ideal.
Here, we say that an ideal $I$ of $R$ is {\em quasi-decomposable} if there exists an $R$-regular sequence $\xx=x_1,\dots,x_n$ in $I$ such that $I/(\xx)$ is decomposable. Examples of rings with quasi-decomposable maximal ideal include the 2-dimensional non-Gorenstein
normal local domains with a rational singularity (see Example~\ref{ex20170916}); more examples of such rings are given in Section~\ref{a}.
Our main result in this direction is Theorem \ref{31}, which especially yields a complete classification of the thick subcategories of the singularity category $\D_{\sg}(R)$ of such $R$:

\begin{intthm}\label{Thm B}
Let $R$ be a Cohen-Macaulay singular local ring whose maximal ideal is quasi-decomposable.
Suppose that on the punctured spectrum $R$ is either locally a hypersurface or locally has minimal multiplicity.
Then taking the singular supports induces a one-to-one correspondence between the thick subcategories of $\D_{\sg}(R)$ and the specialization-closed subsets of $\sing R$.
\end{intthm}

\noindent
Here, $\sing R$ stands for the singular locus of $R$.
The singularity category $\D_{\sg}(R)$, which is also called the stable derived category of $R$, is a triangulated category that has been introduced by Buchweitz \cite{B} and Orlov \cite{O}, and studied deeply and widely so far; see \cite{sg} and references therein.
The {\em singular support} of an object $C$ of $\D_{\sg}(R)$ is defined as the set of prime ideals $\p$ such that $C_\p$ is non-zero in $\D_{\sg}(R_\p)$.
The inverse map of the bijection in the above theorem can also be given explicitly.

There are indeed a lot of examples satisfying the assumptions of Theorems \ref{cor} and \ref{Thm B} (more precisely, Theorems \ref{1} and \ref{31}), and we shall present some of them.

The organization of this paper is as follows.
Section \ref{sec170215a} is devoted to notation, definitions and some basic properties which are used in later sections.
In Section \ref{sec161010b}, we prove a structure result of syzygies of modules over a local ring with decomposable maximal ideal, which gives rise to Theorem \ref{cor} as an immediate corollary.
In Section \ref{a}, for a Cohen-Macaulay local ring with quasi-decomposable maximal ideal, we classify resolving/thick subcategories of module/derived/singularity categories, including Theorem \ref{Thm B}.
Sections \ref{b} and \ref{c} state several other applications of our main results, including the main theorems of \cite{NSW}.

\section{Background and conventions}\label{sec170215a}

This section contains the terminology and some of the definitions and their basic properties that will tacitly be used in this paper. For more details see~\cite{stcm, crs, thd}.

\begin{para}
Throughout this paper, $R$ is a commutative noetherian local ring with maximal ideal $\fm$ and residue field $k$, and all modules are finitely generated.
\end{para}

\begin{para}
Let $M$ be an $R$-module.
All syzygies of $M$ are calculated by using the minimal free resolution of $M$, and the $i$-th syzygy of $M$ is denoted by $\syz_R^i M$.
The minimal number of generators for $M$ is denoted by $\nu_R(M)$.
For an integer $i\ge0$ the $i$th Betti number of $M$ is denoted by $\beta_i^R(M)$; note by definition that $\beta_i^R(M)=\nu_R(\syz_R^iM)$.
\end{para}

\begin{para}
We denote by $\sing(R)$ the \emph{singular locus} of $R$, namely, the set of prime ideals $\p$ of $R$ for which the local ring $R_{\p}$ is singular (i.e. non-regular).
Also, $\ng(R)$ stands for the \emph{non-Gorenstein locus} of $R$, that is, the set of prime ideals $\p$ of $R$ such that $R_\p$ is not Gorenstein.
\end{para}

\begin{para}
We denote the \emph{punctured spectrum} $\spec R\setminus\{\m\}$ of $R$ by $\speco R$.
Recall that $R$ is said to have an \emph{isolated singularity} if $R_\p$ is a regular local ring for all $\p\in \speco R$, in other words, if $\sing R\subseteq\{\m\}$.
\end{para}

\begin{para}
A subset $S$ of $\sing R$ is called {\em specialization-closed}
provided that for prime ideals $\p\subseteq\q$ of $R$ if $\p$ is in $S$, then
so is $\q$.
\end{para}

\begin{para}
Throughout this paper, all subcategories are full and closed under isomorphism.
We denote by $\mod R$ the category of (finitely generated) $R$-modules and by $\PD(R)$ the subcategory of $\mod R$ consisting of modules of finite projective dimension. Also $\cm(R)$ is the subcategory of $\mod R$ consisting of maximal Cohen-Macaulay $R$-modules. By $\add R$ we denote the subcategory of $\mod R$ consisting of direct summands
of finite direct sums of copies of $R$.
\end{para}

\begin{para}
We denote by $\D^{\b}(R)$ the bounded derived category of $\mod R$, and by $\D_{\sg}(R)$ the \emph{singularity category} of $R$, that is, the Verdier
quotient of $\D^{\b}(R)$ by perfect complexes. We denote by $\pi\colon \D^{\b}(R) \to \D_{\sg}(R)$ the canonical functor.

For a subcategory $\X$ of $\D^{\b}(R)$ we denote by $\pi \X$ the subcategory of
$\D_{\sg}(R)$ consisting of objects $M$ such that $M\cong\pi X$ for some $X\in\X$.
For a subcategory $\Y$ of $\D_{\sg}(R)$ we denote by $\pi^{-1}\Y$ the
subcategory of $\D^{\b}(R)$ consisting of objects $N$ such that $\pi N$ is in $\Y$.
\end{para}

\begin{para}
For an object $M\in \D_{\sg}(R)$, the {\em singular support} $\ssupp M$
of $M$ is defined as the set of prime ideals $\p$ of $R$ such that
$M_{\p}\not\cong 0$ in $\D_{\sg}(R_{\p})$.

For a subcategory $\X$ of $\D_{\sg}(R)$, the {\em singular support} $\ssupp \X$ of $\X$ is defined by $\ssupp \X:=\bigcup_{X\in \X}\ssupp X$.

For a subset $S$ of $\spec R$, we denote by $\ssupp^{-1}S$ the subcategory
of $\D_{\sg}(R)$ consisting of objects whose singular supports are
contained in $S$.
\end{para}

\begin{para}
For an $R$-module $M$, the {\em non-free locus} $\nf(M)$ (resp. {\em infinite projective dimension locus} $\ipd(M)$) of $M$ is defined as the set of prime ideals $\p$ of $R$ such that $M_\p$ is non-free (resp. of infinite projective dimension) over $R_\p$.

For a subcategory $\X$ of $\mod R$, the {\em non-free locus} $\nf(\X)$ (resp. {\em infinite projective dimension locus} $\ipd(\X)$) of $\X$ is defined by $\nf(\X):=\bigcup_{X\in\X}\nf(X)$ (resp. $\ipd(\X):=\bigcup_{X\in\X}\ipd(X)$).

For a subset $S$ of $\spec R$, we denote by $\nf^{-1}_\cm(S)$ (resp. $\ipd^{-1}(S)$) the subcategory of $\cm(R)$ (resp. $\mod R$) consisting of modules whose non-free (resp. infinite projective dimension) loci are contained in $S$.
\end{para}

\begin{para}
There are three types of restriction maps $\rest$ that we will consider in this paper. We clarify each one for the convenience of the reader.
\begin{itemize}
\item
For a subcategory $\X$ of $\mod R$, we set $\rest_{\cm(R)}(\X):= \X \cap \cm(R)$.
\item
For a subcategory $\X$ of $\D^{\b}(R)$, we set $\rest_{\mod R}(\X):= \X \cap \mod R$.
\item
For a subcategory $\X$ of $\D_{\sg}(R)$, we set $\rest_{\cm(R)}(\X):=
\pi^{-1}\X\cap \cm(R)$.
\end{itemize}
\end{para}

\begin{para}
Let $\A$ be an abelian category with enough projective objects.
We say that a subcategory $\X$ of $\A$ is \emph{resolving} if it contains the projective objects and is closed under direct
summands, extensions, and kernels of epimorphisms.
Note that a resolving subcategory of $\mod R$ is nothing but a subcategory containing $R$ and closed under direct summands, extensions and syzygies.

For an object $M\in\A$, we denote by $\res M$ the smallest resolving subcategory of $\A$ that contains $M$. This subcategory is called the {\em resolving closure} of $M$.
\end{para}

\begin{para}
Let $\A$ (resp. $\T$) be an abelian category (resp. a triangulated category).
A subcategory $\X$ of $\A$ (resp. $\T$) is called {\em thick} provided
that $\X$ is closed under direct summands, and for any exact sequence
$0\to L\to M\to N\to 0$ in $\A$ (resp. exact triangle $L\to M\to
N\rightsquigarrow$ in $\T$), if two of $L,M,N$ are in $\X$, then so is
the third.
The {\em thick closure} of a subcategory $\X$ of $\A$ (resp. $\T$),
denoted by $\thick_\A\X$ (resp. $\thick_\T\X$), is by definition the
smallest thick subcategory of $\A$ (resp. $\T$) containing $\X$.
Note that $\thick_{\mod R}R=\PD(R)$, and that $\thick_{\D^{\b}(R)}R$
consists of the {\em perfect complexes}, that is, bounded complexes of
free $R$-modules (of finite rank).
For a subcategory $\X$ of $\D^{\b}(R)$, we simply write
$\thick_{\D_{\sg}(R)}\X$ for $\thick_{\D_{\sg}(R)}(\pi\X)$.
\end{para}

\begin{para}
We say that $R$ is a {\em hypersurface} if its $\m$-adic completion $\widehat R$ is isomorphic to the quotient of a regular local ring by an element.
(Hence, by definition, any regular local ring is a hypersurface.)
\end{para}

\begin{para}
Let $R$ be a Cohen-Macaulay local ring.
Then $R$ satisfies the inequality
\begin{equation}\label{eedim}
\e(R)\ge \edim R-\dim R+1,
\end{equation}
where $\e(R)$ and $\edim R$ stand for the multiplicity of $R$ and the embedding dimension of $R$, respectively.
We say that $R$ has {\em minimal multiplicity} (or {\em maximal embedding dimension}) if the equality of \eqref{eedim} holds.
If the residue field $k$ of $R$ is infinite, then this is equivalent to the existence of a system of parameters $\xx=x_1,\dots,x_d$ of $R$ such that $\m^2=\xx\m$; see \cite[Exercise 4.6.14]{BH}.
\end{para}

\begin{para}
Let $R$ be a Cohen-Macaulay local ring.
We say that $\speco R$ is a {\em hypersurface} (resp. has {\em minimal multiplicity}, has {\em finite Cohen-Macaulay representation type}) if the Cohen-Macaulay local ring $R_\p$ is a hypersurface (resp. has minimal multiplicity, has finite Cohen-Macaulay representation type) for every $\p\in\speco R$.
A Cohen-Macaulay local ring with an isolated singularity is a typical example satisfying all of these three conditions, so they are mild conditions.
\end{para}

\section{Local rings with decomposable maximal ideal}\label{sec161010b}

This section is devoted to the proof of {Theorem~\ref{cor}. This theorem is a consequence of Theorem~\ref{1}, which is the main result of this section.

The following fact was established by Ogoma~\cite[Lemma 3.1]{ogoma} and gives a characterization of local rings with decomposable maximal ideal in terms of fiber products.

\begin{fact}\label{fact300117a}
The maximal ideal $\m$ has a direct sum decomposition $\m=I\oplus J$ in which $I,J$ are non-zero ideals of $R$ if and only if $R$ is isomorphic to the non-trivial fiber product $\left(R/I\right)\times_k\left(R/J\right)$, that is, $R$ is isomorphic to the pull-back of the natural surjections $R/I\xra{\pi_I}k\xla{\pi_J}R/J$. This isomorphism is naturally defined by
$r\mapsto (r+I,r+J)$ for $r\in R$. In this case, by~\cite[Remark 3.1]{christensen:gmirlr} we have
the equality
\begin{equation}\label{eq310117b}
\depth R=\min\{\depth R/I,\,\depth R/J,\,1\}.
\end{equation}
In particular, if $\m$ is decomposable and $M$ is an $R$-module, then it follows from the Auslander-Buchsbaum formula that
$\pd_RM\ge2$ if and only if $\pd_R M=\infty$.
%Set $(S,\n,k):=(R/I, \m/I, k)$ and $(T,\l,k):=(R/J, \m/J, k)$.
\end{fact}

\begin{lem}\label{6}
Suppose that $\m$ has a direct sum decomposition $\m=I\oplus J$.
Let $N$ be an $R/I$-module, and let $n:=\nu_{R/I}(N)=\nu_R(N)$.
Then there is an isomorphism $\syz_RN\cong I^{\oplus n}\oplus\syz_{R/I}N$ of $R$-modules.
\end{lem}

\begin{proof}
There is a commutative diagram
$$
\xymatrix{
& 0\ar[d] & 0\ar[d] \\
& I^{\oplus n}\ar@{=}[r]\ar[d] & I^{\oplus n}\ar[d] \\
0\ar[r] & \syz_RN\ar[r]^-\lambda\ar[d]^\zeta & R^{\oplus n}\ar[r]^-\rho\ar[d]^\varepsilon & N\ar[r]\ar@{=}[d] & 0 \\
0\ar[r] & \syz_{R/I}N \ar[r]^-\theta\ar[d] & (R/I)^{\oplus n}\ar[r]^-\varpi\ar[d] & N\ar[r] & 0 \\
& 0 & 0
}
$$
with exact rows and columns.
Now we have
$$
\syz_{R/I}N=\Ker\varpi\subseteq\m (R/I)^{\oplus n}=(\m/I)^{\oplus n}\cong J^{\oplus n}\subseteq R^{\oplus n}.
$$
Call this composite injection $\eta\colon \syz_{R/I}N\to R^{\oplus n}$.
Then we have $\theta=\varepsilon\eta$, and $\rho\eta=\varpi\varepsilon\eta=\varpi\theta=0$.
Hence, there is a homomorphism $\xi\colon \syz_{R/I}N\to\syz_RN$ such that $\eta=\lambda\xi$.
Since $\theta\zeta\xi=\varepsilon\lambda\xi=\varepsilon\eta=\theta$ and $\theta$ is injective, we get $\zeta\xi=1$.
Therefore, the left column in the above diagram splits and we get $\syz_RN\cong I^{\oplus n}\oplus\syz_{R/I}N$.
\end{proof}

\begin{lem}\label{14}
Suppose that $\m$ has a direct sum decomposition $\m=I\oplus J$ such that $J\ne0$. Then the following hold.
\begin{enumerate}[\rm(1)]
\item
If $R/I$ is a discrete valuation ring, then there is an isomorphism $R/I\cong J$ of $R$-modules.
\item
Conversely, if $J$ is a free $R/I$-module, then $R/I$ is a discrete valuation ring.
\end{enumerate}
\end{lem}

\begin{proof}
Note that $\m/I\cong J$ as $R$-modules.

(1) By assumption we have $\m/I\cong R/I$. Hence, the assertion follows.

(2) By assumption, $\m/I$ is a free $R/I$-module.
Hence, the global dimension of $R/I$ is at most one.
If it is zero, then $R/I$ is a field and we have $I=\m$, which contradicts $J$ being non-zero.
Thus, $R/I$ is a discrete valuation ring.
\end{proof}

The following lemma is shown by a similar argument as in \cite[Lemma 2.5]{NSW}.

\begin{lem}\label{lem020117a}
Let $I$ be a
proper ideal of $R$. Let $M$ be an $R$-module with $\pd_RM\ge2$. If $\syz_R^2M$ is $R/I$-free, then $R/I$ has positive depth.
\end{lem}

\begin{proof}
By assumption there are injections $R/I\to\syz_R^2M\to\m R^{\oplus e}=\m^{\oplus e}$
for some $e>0$, and call this composition $\phi\colon R/I\to\m^{\oplus e}$.
Set $\xi=\phi(\overline 1)\in\m^{\oplus e}$.
If $R/I$ has depth zero, then it contains a non-zero socle $\overline s$, and we have $\phi(\overline s)=s\xi=0$, which contradicts the injectivity of the map $\phi$.
Hence, $R/I$ has positive depth.
\end{proof}

\begin{prop}\label{prop020117a}
Suppose that $\m$ has a direct sum decomposition $\m=I\oplus J$ in which $I,J$ are non-zero ideals of $R$. Let $M$ be an $R$-module with $\pd_RM\ge2$. If either $R/I$ or $R/J$ has depth zero, then $\syz_R^2M$ is neither $R/I$-free nor $R/J$-free.
\end{prop}

\begin{proof}
Assume that $\depth R/I=0$. (The case where $\depth R/J=0$ is handled similarly.)
It follows from Lemma~\ref{lem020117a} that $\syz_R^2M$ is not $R/I$-free.
Suppose that $\syz_R^2M$ is $R/J$-free.
Then again by Lemma~\ref{lem020117a} the local ring $R/J$ has positive depth.
There is an exact sequence
$0 \to (R/J)^{\oplus a} \to R^{\oplus b} \xrightarrow{(c_{ij})} R^{\oplus d} \to M \to 0$
where each $c_{ij}$ is an element of $\m$ and $a,b,d$ are positive integers.
Applying the functor $\Hom_R(k,-)$ to this, we have an exact sequence
$$
0 \to \Hom_R(k,(R/J)^{\oplus a}) \to \Hom_R(k,R^{\oplus b}) \xrightarrow{0} \Hom_R(k,R^{\oplus d}).
$$
As $R/J$ has positive depth, $\Hom_R(k,(R/J)^{\oplus a})=0$.
We see from this exact sequence that $\Hom_R(k,R^{\oplus b})=0$, and hence
$\depth R>0$.
This is a contradiction because by~\eqref{eq310117b} we must have $\depth R=0$.
Thus, $\syz_R^2M$ is not $R/J$-free.
\end{proof}

%Now we can prove Theorem~\ref{1}.
%In this proof, for an $R$-module $M$, the notation $\beta_2^R(M)$ is the second Betti number of $M$, that is, $\beta_2^R(M)=\rank_k\left(\Ext^2_R(M,k)\right)$.

Theorem~\ref{cor} is a consequence of the following result, which is the main result of this section.

\begin{thm}\label{1}
Let $(R,\fm)$ be a local ring, and suppose that there is a non-trivial direct sum decomposition $\m=I\oplus J$.
Let $M$ be an $R$-module with $\pd_RM\ge2$.
\begin{enumerate}[\rm(1)]
\item
One of the following holds true.
\begin{enumerate}[\rm(i)]
\item
$\m$ is a direct summand of $\syz_R^3M$ or $\syz_R^4M$.
%\item
%$\m$ is a direct summand of $\syz_R^4M$.
\item
$\m$ is a direct summand of $\syz_R^5M$, and $\syz_R^2M$ is $R/I$-free.
\item
$\m$ is a direct summand of $\syz_R^5M$, and $\syz_R^2M$ is $R/J$-free.
\item
$\m$ is a direct summand of
$\Omega_R^3M\oplus\Omega_R^4M$, $R/I$ is a discrete valuation ring, and $\syz_R^2M$ is $R/J$-free.
\item
$\m$ is a direct summand of
$\Omega_R^3M\oplus\Omega_R^4M$, $R/J$ is a discrete valuation ring, and $\syz_R^2M$ is $R/I$-free.
\end{enumerate}
\item
If either $R/I$ or $R/J$ has depth $0$, then none of {\rm(ii)}-{\rm(v)} occurs.
%\item
%If {\rm(iii)} holds, then $\m$ is a direct summand of $\syz_R^3M\oplus\syz_R^4M$.
\end{enumerate}
\end{thm}

\begin{proof}
(1) It follows from~\cite{dress} (see also~\cite[Proposition 4.2]{moore}) that
$$
\syz_R^2M\cong A\oplus B,\qquad
\syz_R^3M\cong V\oplus W,\qquad
\syz_R^4M\cong X\oplus Y
$$
for some $R/I$-modules $A,V,X$ and $R/J$-modules $B,W,Y$.
By Lemma \ref{6}, we have\footnote{We use $\nu$ instead of $\nu_R$ to keep the notation simpler.}
\begin{align*}
\syz_RA&\cong I^{\oplus\nu(A)}\oplus\syz_{R/I}A,& \syz_RV&\cong I^{\oplus\nu(V)}\oplus\syz_{R/I}V,&\! \syz_RX&\cong I^{\oplus\nu(X)}\oplus\syz_{R/I}X,\\
\syz_RB&\cong J^{\oplus\nu(B)}\oplus\syz_{R/J}B,& \syz_RW&\cong J^{\oplus\nu(W)}\oplus\syz_{R/J}W,&\! \syz_RY&\cong J^{\oplus\nu(Y)}\oplus\syz_{R/J}Y.
\end{align*}

We consider the following five cases. Note that since $\pd_RM\ge2$, the second Betti number $\beta:=\beta_2^R(M)$ is non-zero.

Case 1: $A\ne0\ne B$. In this case, $I$ and $J$ are direct summands of $\syz_RA$ and $\syz_RB$, respectively.
Hence, $\m=I\oplus J$ is a direct summand of $\syz_RA\oplus\syz_RB=\syz_R^3M$.

Case 2: $A=0=V$. In this case, $\syz_R^2M=B$ and $\syz_R^3M=W$ are both $R/J$-modules, and annihilated by the ideal $J$ of $R$.
Hence, by the short exact sequence
$0\to \syz_R^3M \to R^{\oplus \beta} \to \syz_R^2M \to 0$
we observe that $J^2$ annihilates $R^{\oplus \beta}$.
Since $\beta\neq 0$ we get $J^2=0$. Therefore, $\m J=(I+J)J=0$.
This implies that $J$ is a $k$-vector space.
Since $B=\syz_R^2M\ne0$, the ideal $J$ is a direct summand of $\syz_RB=\syz_R^3M$.
Thus $k$ is a direct summand of $\syz_R^3M$, and $\m$ is a direct summand of $\syz_R^4M$.

Case 3: $B=0=W$. This case is treated similarly to Case 2, and we
see that $\m$ is a direct summand of $\syz_R^4M$.

Case 4: $A=0\ne V$. In this case, we must have $B\ne0$, and hence,
$n:=\nu(B)>0$.
If $W\ne0$, then $V\ne0\ne W$, and Case 1 shows that $\m$ is a direct summand of $\syz_R^4M$.
Thus, we may assume that $W=0$.
Then we have $\syz_R^2M\cong B$ and
$$
V\cong \syz_R^3M\cong \syz_RB\cong J^{\oplus n}\oplus\syz_{R/J}B.
$$
Since $V$ is annihilated by $I$, so is the direct summand $\syz_{R/J}B$.
As $\syz_{R/J}B$ is also annihilated by $J$, and by $\m=I\oplus J$, we see that $\syz_{R/J}B$ is a $k$-vector space.
If $\syz_{R/J}B$ is non-zero, then $k$ is a direct summand of $\syz_{R/J}B$, and hence, of $\syz_R^3M$.
In this case $\m$ is a direct summand of $\syz_R^4M$, and we are done.
Thus, we may assume that $\syz_{R/J}B=0$. In this case, $\syz_R^2M\cong B\cong (R/J)^{\oplus n}$ and we have
\begin{equation}\label{eq010217a}
V\cong \syz_R^3M\cong \syz_RB\cong J^{\oplus n}.
\end{equation}
If $Y=0$, since we already assumed that $W=0$, Case 3 shows that $\m$ is a direct summand of $\syz_R^5M$.
Hence, we may assume $Y\ne0$.
If $X\ne0\ne Y$, then Case 1 shows that $\m$ is a direct summand of $\syz_R^5M$.
Therefore, we may assume $X=0$.
Now we have
$$
Y\cong \syz_R^4M\cong \syz_RV\cong I^{\oplus m}\oplus\syz_{R/I}V
$$
where $m:=\nu(V)$.
We make a similar argument as above.
Since $Y$ is annihilated by $J$, so is $\syz_{R/I}V$, which is also annihilated by $I$.
Hence, $\syz_{R/I}V$ is a $k$-vector space.
If $\syz_{R/I}V$ is non-zero, then $k$ is a direct summand of $\syz_{R/I}V$, and thus, of $\syz_R^4M$.
In this case $\m$ is a direct summand of $\syz_R^5M$.
So, we may assume $\syz_{R/I}V=0$.
Under this assumption, $V$ is a free $R/I$-module and we get $V\cong (R/I)^{\oplus m}$.
Therefore, by~\eqref{eq010217a} we have
$J^{\oplus n}\cong (R/I)^{\oplus m}$.
Since $n>0$ we see that $J$ is a free $R/I$-module.
Now Lemma~\ref{14}(2) implies that $R/I$ is a discrete valuation ring.
Then $J$ is isomorphic to $R/I$ by Lemma \ref{14}(1). We obtain
isomorphisms $\syz_R^3M\cong J^{\oplus n}\cong(R/I)^{\oplus n}$ and
$\syz_R^4M\cong I^{\oplus n}$. Therefore, $\m=I\oplus J$ is a direct
summand of $\syz_R^3M\oplus\syz_R^4M$.

Case 5: $B=0\ne W$. This case is treated similarly to Case 4, and we see that one of the following holds: $\m$ is a direct summand of $\syz^4M$, or $\syz_R^2M$ is $R/I$-free and $\m$ is a direct summand of $\syz_R^5M$, or $\syz_R^2M$ is $R/I$-free and $R/J$ is a discrete valuation ring and $\m$ is a direct summand of
$\syz_R^3M\oplus\syz_R^4M$.

(2) This part follows from (1) and Proposition~\ref{prop020117a}.
\end{proof}

In the remainder of this section, we present several examples showing that each of the cases (i)-(v) in Theorem~\ref{1} occurs and these conditions are independent.

\begin{convention}
In the following examples, $k$ will be a field as before, and $X,Y, Z$ will be indeterminates over $k$. We will use the lower-case letters $x, y, z$ to
represent the residues of the variables $X, Y, Z$ in the ring $R$.
\end{convention}

We start by providing examples satisfying (iv) and (v), but not (i)-(iii).

\begin{ex}\label{11}
Consider the fiber product
$$
R=k[[X]]\times_kk[[Y]]\cong\frac{k[[X,Y]]}{(XY)}
$$
and set $I=(x)$ and $J=(y)$. We then have $\m=I\oplus J$.
We also have $R/J\cong k[[x]]$ and $R/I\cong k[[y]]$, which are discrete valuation rings.
Take $M=R/I$.
Then
$$
\syz_R^iM\cong
\begin{cases}
R/J & \text{if $i\ge0$ is odd},\\
R/I & \text{if $i\ge0$ is even}.
\end{cases}
$$
In particular, $\syz_R^2M$ is $R/I$-free, and $\syz_R^3M\oplus\syz_R^4M\cong\m$. Thus, (v) holds.
Moreover, $\m$ is not a direct summand of $\syz_R^iM$ for all $i\ge0$. Hence, none of (i)-(iii) holds.

Similarly, if we set $M=R/J$, then (iv) holds, but none of (i)-(iii) holds.
\end{ex}

The following is also an example that satisfies (iv), but not (i)-(iii).
Comparing with Example~\ref{11}, this example has the advantage that $R/J$ is not a discrete valuation ring, hence, it does not satisfy (v) as well. (By symmetry, one can construct an example that satisfies (v), but not (i)-(iv).)

\begin{ex}\label{10}
Consider the fiber product
$$
R=k[[X]]\times_kk[[Y,Z]]\cong\frac{k[[X,Y,Z]]}{(XY,XZ)}.
$$
Setting $I=(y,z)$ and $J=(x)$, we have $\m=I\oplus J$.
The ring $R/I=k[[x]]$ is a discrete valuation ring, while $R/J=k[[y,z]]$ is not.
One can easily check (by hand) that the minimal free resolution of the $R$-module $R/(y)$ has the form:
$$
\cdots\to R^5\xra{\left(
\begin{smallmatrix}
y&z&0&0&0\\
0&0&y&z&0\\
0&0&0&0&x
\end{smallmatrix}
\right)
}
R^3\xra{
\left(
\begin{smallmatrix}
x&0&z\\
0&x&-y
\end{smallmatrix}
\right)
}
R^2\xra{(y,z)}R\xra{x}R\xra{y} R\to R/(y)\to 0
$$
Note that it contains a minimal free presentation of $I$ and $\syz_RI\cong(R/I)^{\oplus2}\oplus R/(x)$.
Let $M$ be the {\em (Auslander) transpose} of $I$, that is, the $R$-module $M$ defined by the following exact sequence
$$
0 \to I^* \to R^2 \xrightarrow{
\left(
\begin{smallmatrix}
x&0\\
0&x\\
z&-y
\end{smallmatrix}
\right)
} R^3 \to M \to 0
$$
where $I^*:=\Hom_R(I,R)$.
Then $\syz_R^2M=I^*$.
From the natural short exact sequence $0\to I \to R\to R/I\to 0$, we have an exact sequence
$$
0 \to (R/I)^* \xrightarrow{f} R^* \to I^* \to \Ext_R^1(R/I,R)\to0.
$$
Note that the map $f$ can be described by the inclusion map $J=(0:I)\to R$, and hence, there is an exact sequence
\begin{equation}\label{2}
0 \to R/J \to I^* \to  \Ext_R^1(R/I,R)\to0.
\end{equation}
We claim that $\Ext_R^1(R/I,R)=0$.
Indeed, from the natural short exact sequence
$0 \to R \to \left(R/I\right)\oplus \left(R/J\right) \to k \to 0$
we get an exact sequence
$$
\left(R/I\right)^*\oplus \left(R/J\right)^* \xrightarrow{g} R^* \to \Ext_R^1(k,R) \to \Ext_R^1(\left(R/I\right)\oplus \left(R/J\right),R) \to0.
$$
Since the map $g$ can be described by the inclusion map $$\m=J\oplus I=(0:I)\oplus(0:J)\to R$$ its cokernel is isomorphic to $k$.
The element $x-y\in R$ is $R$-regular and annihilates $k$.
Using \cite[Lemma 3.1.16]{BH} and the isomorphism $R/(x-y)\cong k[[y,z]]/(y^2,yz)$, we easily see that
$\Ext_R^1(k,R)\cong\Hom_{R/(x-y)}(k,R/(x-y))\cong k$.
Thus, we obtain $\Ext_R^1(\left(R/I\right)\oplus \left(R/J\right),R)=0$, which implies that $\Ext_R^1(R/I,R)=0$, as we claimed.

This claim and \eqref{2} imply that
$\syz_R^2M=I^*\cong R/J$.
Combining this with the isomorphism $\syz_R^3M\oplus\syz_R^4M\cong\m$ shows that (iv) holds.

On the other hand, none of the syzygies
$\syz_R^3M=J$, $\syz_R^4M=I$, and $\syz_R^5M=\syz_RI=(R/I)^2\oplus \left(R/J\right)$
contain $\m$ as a direct summand.
In fact, taking into account the minimal numbers of generators, we see that $\m$ is not a direct summand of $J$ or $I$. Also, $\m\cong(R/I)^{\oplus 2}\oplus \left(R/J\right)$
if we assume that $\m$ is a direct summand of $\syz_RI$.
Then, however, localization at the prime ideal $I$ yields $R_I=\m R_I\cong\left(R_I/IR_I\right)^{\oplus2}$
which contradicts the fact that the local ring $R_I$ is indecomposable as a module over itself.
Consequently, none of (i)-(iii) and (v) holds.
\end{ex}

The following is an example that satisfies (ii), but not (i) and (iii)-(v). (By symmetry, one can construct an example that satisfies (iii), but not (i)-(ii) and (iv)-(v).)

\begin{ex}\label{17}
Let $R,I,J$ be as in Example \ref{10}, and let $M=R/(y)$.
Then
$$
\syz_R^3M=I,\qquad
\syz_R^4M=(R/I)^{\oplus2}\oplus (R/J),\qquad
\syz_R^5M=I^{\oplus2}\oplus J.
$$
By the argument in Example \ref{10}, $\m$ is not a direct summand of $\syz_R^3M$ or $\syz_R^4M$, while it is a direct summand of $\syz_R^5M$.
%Therefore, (i) does not hold, but (ii) holds.
We also have $\syz_R^2M=R/I$, which is not $R/J$-free, and $R/J=k[[y,z]]$ is not a discrete valuation ring.
Thus, (i) and (iii)-(v) do not hold.
\end{ex}

The next examples satisfy (i), but they do not satisfy (ii)-(v).

\begin{ex}
(a) Consider the fiber product
$$
R=\frac{k[[X]]}{(X^2)}\times_kk[[Y]]\cong\frac{k[[X,Y]]}{(X^2,XY)}
$$
and set $I=(x)$ and $J=(y)$. Then we have $\m=I\oplus J$.
Since $R/J$ has depth zero, it follows from Theorem \ref{1} that none of (ii)-(v) holds.
Let $M=R/J$.
The minimal free resolution of $M$ is
$$
\cdots\to
R^3 \xra{\left(
\begin{smallmatrix}
x&y&0\\
0&0&x
\end{smallmatrix}
\right)}
R^2\xra{(x,y)}R\xra{x}R\xra{y}R\to M\to 0.
$$
It follows from this that $\syz_R^3M=\m\cong k\oplus \left(R/I\right)$, and we observe that $\m$ is a direct summand of $\syz_R^iM$ for all $i\ge3$.
In particular, statement (i) holds true.

We can also construct an example satisfying (i) but not (ii)-(v) using the ring of Example \ref{10}.

(b) Let $R,I,J$ be as in Example \ref{10}, and let $M=(y)\cong R/(x)$.
Then Example \ref{10} shows that $\m$ is not a direct summand of $\syz_R^3M=(R/I)^{\oplus2}\oplus (R/J)$, but it is a direct summand of $\syz_R^4M=I^{\oplus2}\oplus J$.
%Hence, half of (i) holds, while the other half does not.
As $I^2=(y^2,yz,z^2)$ is non-zero, $\syz_R^2M=I$ is not $R/I$-free.
Since $R/J\cong k[[y,z]]$ is not a discrete valuation ring, $\syz_R^2M=I$ is not $R/J$-free by Lemma \ref{14}(2).
Thus, (ii)-(v) do not hold.
\end{ex}

\section{Local rings with quasi-decomposable maximal ideal}\label{a}

This section contains the proof of Theorem~\ref{Thm B}. Recall that the definitions and terminology have been introduced in Section~\ref{sec170215a}.

\begin{defn}
If there exists an $R$-sequence $\xx$ of length $n\ge 0$ such that $\m/(\xx)$ is decomposable, then we say that $\m$ is {\em quasi-decomposable} (with $\xx$).
\end{defn}

If $\m$ is decomposable, then of course $\m$ is quasi-decomposable.
However, the converse is far from being true.
We shall give several examples of a local ring with indecomposable quasi-decomposable maximal ideal at the end of this section.

We prove the following lemmas for the future use.

\begin{lem}\label{ooor}
Let $M$ be an $R$-module, and let $\xx$ be an $R$-sequence of length $n\ge 0$ annihilating $M$.
Then for all $t\ge0$ and $u\ge n$ there exists $v\ge0$ such that
$$
\syz_R^u\syz_{R/(\xx)}^tM\cong\syz_R^{u+t}M\oplus R^{\oplus v}.
$$
\end{lem}

\begin{proof}
There is an exact sequence
$$
0 \to \syz_{R/(\xx)}^tM \to F_{t-1} \to F_{t-2} \to \cdots \to F_1 \to F_0 \to M \to 0
$$
of $R/(\xx)$-modules with each $F_i$ free over $R/(\xx)$.
Taking the $u$-th syzygies over $R$, we get an exact sequence
\begin{align*}
0 \to \syz_R^u\syz_{R/(\xx)}^tM &\to \syz_R^uF_{t-1}\oplus R^{\oplus a_{t-1}} \to \cdots \\
&\to \syz_R^uF_1\oplus R^{\oplus a_1} \to \syz_R^uF_0\oplus R^{\oplus a_0} \to \syz_R^uM \to 0
\end{align*}
for some $a_i\ge0$.
As $u\ge n$, each $\syz_R^uF_i$ is $R$-free.
This exact sequence then says that $\syz_R^u\syz_{R/(\xx)}^tM$ is isomorphic to the direct sum of $\syz^t\syz_R^uM=\syz^{u+t}_RM$ and some free $R$-module, as desired.
\end{proof}

\begin{lem}\label{res}
Let $M$ be an $R$-module, and let $\xx$ be an $M$-sequence of length $n\ge 0$.
Then $\syz_R^i(M/\xx M)$ belongs to the resolving closure $\res M$ of $M$ for all $i\ge n$.
\end{lem}

\begin{proof}
Since $\xx$ is an $M$-sequence, the Koszul complex $\k(\xx,M)$
is acyclic with $\HH_0(\k(\xx,M))=M/\xx M$.
Each $\k(\xx,M)_i$ is a direct sum of copies of $M$, which belongs to $\res M$.
Hence, $\syz_R^n(M/\xx M)$ is in $\res M$ by \cite[Lemma 4.3]{stcm}. Therefore, $\syz_R^i(M/\xx M)$ is in $\res M$ for all $i\ge n$.
\end{proof}

\begin{lem}\label{34}
Let $R$ be a $d$-dimensional Cohen-Macaulay local ring with quasi-decomposable maximal ideal.
Let $\X$ be a resolving subcategory of $\mod R$ contained in $\cm(R)$.
If $\X$ contains a non-free $R$-module, then $\X$ also contains $\syz_R^dk$.
\end{lem}

\begin{proof}
Let $\xx$ be an $R$-sequence of length $n\ge 0$ such that $\m/(\xx)$ is decomposable, and let $X\in\X$ be a non-free $R$-module.
Then $X$ is maximal Cohen-Macaulay, and $\xx$ is an $X$-sequence. Hence, $\pd_{R/(\xx)}X/\xx X=\pd_RX=\infty$ by~\cite[Lemma 1.3.5]{BH}.
Theorem \ref{cor} implies that $\syz_{R/(\xx)}k=\m/(\xx)$ is a direct summand of
$$
\syz_{R/(\xx)}^3(X/\xx X)\oplus\syz_{R/(\xx)}^4(X/\xx X)\oplus\syz_{R/(\xx)}^5(X/\xx X)=Y/\xx Y
$$
where $Y:=\syz_R^3X\oplus\syz_R^4X\oplus\syz_R^5X\in\X$.

If $n\le d-1$, then by Lemma \ref{ooor} we have $\syz_R^{d-1}\syz_{R/(\xx)}k\cong\syz_R^dk\oplus F$ for some free $R$-module $F$.
The module $\syz_R^{d-1}\syz_{R/(\xx)}k$ is a direct summand of $\syz_R^{d-1}(Y/\xx Y)$.
As $Y$ is maximal Cohen-Macaulay, $\xx$ is a $Y$-sequence.
It follows from Lemma \ref{res} that $\syz_R^{d-1}(Y/\xx Y)$ is in $\X$.
Hence, so is $\syz_R^{d-1}\syz_{R/(\xx)}k$, and therefore, $\syz_R^dk$ is in $\X$.

Now let us consider the case $n=d$.
Similarly as above, we see that $\syz_R^d\syz_{R/(\xx)}k\cong\syz_R^{d+1}k\oplus F'$ for some free $R$-module $F'$, that $\syz_R^d\syz_{R/(\xx)}k$ is a direct summand of $\syz_R^d(Y/\xx Y)$, and that $\syz_R^d(Y/\xx Y)$ is in $\X$.
We also see that $\syz_R^{d+1}k$ belongs to $\X$.
Since $\depth_R \syz_R^{d+1}k=d$, it follows from~\cite[Proposition 4.2]{crspd} that $\syz_R^dk$ is in $\res(\syz_R^{d+1}k\oplus \syz_R^{d+1}k)=\res\syz_R^{d+1}k$. Hence, $\syz_R^dk$ is in $\X$, as desired.
\end{proof}

Now we can state and prove the main result of this section.

\begin{thm}\label{31}
Let $(R,\m)$ be a $d$-dimensional Cohen-Macaulay singular local ring and assume that $\m$ is quasi-decomposable.
\begin{enumerate}[\rm(1)]
\item
Suppose that $\speco R$ is a hypersurface or has minimal multiplicity.
Then one has the following commutative diagram of mutually inverse bijections:
$$
\xymatrix{
{\left\{
\begin{matrix}
\text{Resolving subcategories of}\\
\text{$\mod R$ contained in $\cm(R)$}
\end{matrix}
\right\}}
\ar@<.7mm>[r]^-\nf\ar@{=}[d] &
{\left\{
\begin{matrix}
\text{Specialization-closed}\\
\text{subsets of $\sing R$}
\end{matrix}
\right\}}
\ar@<.7mm>[l]^-{\nf^{-1}_\cm}\ar@<.7mm>[d]^{\ipd^{-1}} \\
{\left\{
\begin{matrix}
\text{Thick subcategories of}\\
\text{$\cm(R)$ containing $R$}
\end{matrix}
\right\}}
\ar@<.7mm>[r]^{\thick_{\mod R}} \ar@<.7mm>[d]^{\thick_{\D_{\sg}(R)}}&
{\left\{
\begin{matrix}
\text{Thick subcategories of}\\
\text{$\mod R$ containing $R$}
\end{matrix}
\right\}}
\ar@<.7mm>[l]^{\rest_{\cm(R)}}\ar@<.7mm>[u]^\ipd \ar@<.7mm>[d]^{\thick_{\D^{\b}(R)}}\\
{\left\{
\begin{matrix}
\text{Thick subcategories of}\\
\text{$\D_{\sg}(R)$}
\end{matrix}
\right\}}
\ar@<.7mm>[r]^{\pi^{-1}}  \ar@<.7mm>[u]^{\rest_{\cm(R)}}&
{\left\{
\begin{matrix}
\text{Thick subcategories of}\\
\text{$\D^{\b}(R)$ containing $R$}
\end{matrix}
\right\}}
\ar@<.7mm>[l]^{\pi}\ar@<.7mm>[u]^{\rest_{\mod R}}
}
$$
\item
Assume that $R$ is excellent and has a canonical module $\omega$.
Suppose that $\speco R$ has finite Cohen-Macaulay representation type.
Then one has the following commutative diagram of mutually inverse bijections:
$$
\xymatrix{
{\left\{
\begin{matrix}
\text{Resolving subcategories}\\
\text{of $\mod R$ contained in}\\
\text{$\cm(R)$ and containing $\omega$}
\end{matrix}
\right\}}
\ar@<.7mm>[r]^-\nf\ar@{=}[d] &
{\left\{
\begin{matrix}
\text{Specialization-closed}\\
\text{subsets of $\sing R$}\\
\text{containing $\ng R$}
\end{matrix}
\right\}}
\ar@<.7mm>[l]^-{\nf^{-1}_\cm}\ar@<.7mm>[d]^{\ipd^{-1}} \\
{\left\{
\begin{matrix}
\text{Thick subcategories of}\\
\text{$\cm(R)$ containing}\\
\text{$R$ and $\omega$}
\end{matrix}
\right\}}
\ar@<.7mm>[r]^{\thick_{\mod R}} \ar@<.7mm>[d]^{\thick_{\D_{\sg}(R)}}&
{\left\{
\begin{matrix}
\text{Thick subcategories of}\\
\text{$\mod R$ containing}\\
\text{$R$ and $\omega$}
\end{matrix}
\right\}}
%\ar@<.7mm>[l]^{\rest_{\cm(R)}}\ar@<.7mm>[u]^\ipd
\ar@<.7mm>[l]^{\rest_{\cm(R)}}\ar@<.7mm>[u]^\ipd \ar@<.7mm>[d]^{\thick_{\D^{\b}(R)}}\\
{\left\{
\begin{matrix}
\text{Thick subcategories of}\\
\text{$\D_{\sg}(R)$ containing $\omega$}
\end{matrix}
\right\}}
\ar@<.7mm>[r]^{\pi^{-1}}  \ar@<.7mm>[u]^{\rest_{\cm(R)}}&
{\left\{
\begin{matrix}
\text{Thick subcategories of}\\
\text{$\D^{\b}(R)$ containing}\\
\text{$R$ and $\omega$}
\end{matrix}
\right\}}
\ar@<.7mm>[l]^{\pi}\ar@<.7mm>[u]^{\rest_{\mod R}}
}
$$
\end{enumerate}
\end{thm}

\begin{proof}
Note that $\ng(R)=\nf(\omega)=\ipd(\omega)$ for part (2). Hence, by~\cite[Lemma 2.6]{crs} we observe that all the twelve maps in Theorem~\ref{31} are well-defined, and that
\begin{equation}\label{nf}
\nf(\nf^{-1}_\cm(W))=W,\qquad
\ipd(\ipd^{-1}(W))=W
\end{equation}
for each specialization-closed subset $W$ of $\spec R$ contained in $\sing R$.

We first show that
\begin{equation}\label{tr}
\X=\thick_{\mod R}(\rest_{\cm(R)}\X)=\thick_{\mod R}(\X\cap\cm(R))
\end{equation}
for each thick subcategory $\X$ of $\mod R$ containing $R$.
Since $\X$ is thick and contains $\X\cap\cm(R)$, we see that $\X$ contains $\thick_{\mod R}(\X\cap\cm(R))$.
Take any module $X\in\X$.
Then there is an exact sequence $0 \to \syz^dX\to F_{d-1}\to\cdots\to F_0\to X\to0$
of $R$-modules with each $F_i$ free.
Note that $\syz^dX$ and $F_i$ for $0\le i\le d-1$ are in $\X\cap\cm(R)$. Thus, these modules belong to $\thick_{\mod R}(\X\cap\cm(R))$.
Therefore, $X$ is also in $\thick_{\mod R}(\X\cap\cm(R))$, and we have the equalities in~\eqref{tr}.

Next we show that
\begin{equation}\label{it}
\ipd(\thick_{\mod R}\X)=\nf(\X)
\end{equation}
for any subcategory $\X$ of $\cm(R)$.
It is straightforward to check that $$\ipd(\thick_{\mod R}\X)=\ipd(\X)\subseteq\nf(\X).$$
Let $\p$ be a prime ideal such that $X_\p$ is a non-free $R_\p$-module for some $X\in\X$.
Since the $R_\p$-module $X_\p$ is maximal Cohen-Macaulay, it has infinite projective dimension.
Hence, $\nf(\X)$ is contained in $\ipd(\X)$. Therefore, we have the equality in~\eqref{it}.

In view of \eqref{nf}, \eqref{tr} and \eqref{it}, we show the following three statements.
\begin{enumerate}[(a)]
\item
$\nf^{-1}_{\cm}\cdot\nf=1$.
\item
$\ipd^{-1}\cdot\ipd=1$.
\item
The equalities in the diagrams hold.
\end{enumerate}

For (a) fix a resolving subcategory $\X$ of $\mod R$ contained in $\cm(R)$.
If $\X=\add R$, then $\nf(\X)$ is empty, and we have $\nf^{-1}_\cm(\nf(\X))=\add R=\X$.
Thus, we may assume that $\X\ne\add R$.
Then $\X$ contains a non-free $R$-module, and it follows from Lemma \ref{34} that $\X$ contains $\syz^dk$.
If $\speco R$ is a hypersurface (resp. has minimal multiplicity), then by virtue of \cite[Theorem 5.13]{stcm} (resp. \cite[Theorem 5.6]{crs}) we obtain the equality $\nf^{-1}_\cm(\nf(\X))=\X$.
If $R$ is excellent, $\speco R$ has finite Cohen-Macaulay representation type, and $\X$ contains a canonical module, then the equality $\nf^{-1}_\cm(\nf(\X))=\X$ follows from \cite[Theorem 6.11 and Corollary 6.12]{crs}.

For (b) let $\X$ be a thick subcategory of $\mod R$ containing $R$.
Then $\X$ is contained in $\ipd^{-1}(\ipd(\X))$.
Now for an $R$-module $M$ in $\ipd^{-1}(\ipd(\X))$ we claim that
$$
\nf(\syz_R^dM)\subseteq\nf(\X\cap\cm(R)).
$$
Indeed, let $\p$ be a prime ideal such that $(\syz_R^dM)_\p$ is non-free.
Then the $R_\p$-module $\syz_{R_\p}^d(M_\p)$ is non-free, and $M_\p$ has infinite projective dimension.
Hence, $\p$ is in $\ipd(M)$, which is contained in $\ipd(\X)$. This implies that $X_\p$ has infinite projective dimension for some $X$ in $\X$.
Therefore, $(\syz_R^dX)_\p$ is non-free, which says that $\p$ belongs to $\nf(\syz_R^dX)$.
Since $\syz_R^dX$ is in $\X\cap\cm(R)$, the claim follows.

This claim implies that $\syz^dM$ belongs to $\nf^{-1}_\cm(\nf(\X\cap\cm(R)))$.
Note that $\X\cap\cm(R)$ is a resolving subcategory contained in $\cm(R)$.
In each of the cases where
$\speco R$ is a hypersurface, where $\speco R$ has minimal multiplicity, and where $R$ is excellent, $\speco R$ has finite Cohen-Macaulay representation type and $\X$ contains a canonical module,
it follows from (a) that $\nf^{-1}_\cm(\nf(\X\cap\cm(R)))=\X\cap\cm(R)$.
Hence, $\syz^dM$ is in $\X$, and so is $M$ as $\X$ is thick and contains $R$.

Finally, for (c) let $\X$ be a resolving subcategory of $\mod R$ contained in $\cm(R)$.
By (a)
we have $\X=\nf^{-1}_\cm(\nf(\X))$; setting $W=\nf(\X)$, we have $\X=\nf^{-1}_\cm(W)$.
Let $0 \to A \to B \to C \to 0$
be an exact sequence of maximal Cohen-Macaulay $R$-modules.
If $A$ and $B$ are in $\X$, then $\nf(A)$ and $\nf(B)$ are contained in $W$, and we see that $\nf(C)$ is also contained in $W$ by the Auslander-Buchsbaum formula.
Hence $C$ is in $\X$, which shows that $\X$ is a thick subcategory of $\cm(R)$. The proof of (c) is now complete.

The one-to-one correspondence $(\rest_{\mod R},\thick_{\D^{\b}(R)})$ in (1) and (2)
follows from \cite[Theorem 1]{KS}, while the one-to-one correspondence
$(\pi^{-1},\pi)$ is shown by noting that for two objects
$M,N\in\D^{\b}(R)$ with $\pi M\cong\pi N$ one has
$M\in\thick_{\D^{\b}(R)}\{R,N\}$.

It is straightforward that
$\rest_{\cm(R)}\X=\rest_{\cm(R)}(\rest_{\mod R}(\pi^{-1}\X))$ for a
(thick) subcategory $\X$ of $\D_{\sg}(R)$.
Let $\Y$ be a thick subcategory of $\cm(R)$ containing $R$.
Then $\pi(\thick_{\D^{\b}(R)}(\thick_{\mod R}\Y))=\pi(\thick_{\D^{\b}(R)}\Y)$
is a thick subcategory of $\D_{\sg}(R)$ containing $\pi\Y$.
If $\Z$ is a thick subcategory of $\D_{\sg}(R)$ containing $\pi\Y$, then
$\Y$ is contained in $\pi^{-1}\Z$, and so is $\thick_{\D^{\b}(R)}\Y$.
It follows that $\pi(\thick_{\D^{\b}(R)}\Y)=\thick_{\D_{\sg}(R)}\Y$.
Thus, the lower squares in the diagrams in (1) and (2) are commutative.
This completes the proof of Theorem~\ref{31}.
\end{proof}

Our Theorem~\ref{Thm B} can now be deduced from Theorem~\ref{31}.

\begin{para}[\emph{Proof of Theorem~\ref{Thm B}}]\label{para22022017a}
Let $\X$ be a thick subcategory of $\D_{\sg}(R)$, and let $W$ be a
specialization-closed subset of $\sing R$.
According to Theorem \ref{31}, it suffices to show the equalities:
\begin{enumerate}[(a)]
\item
$\ipd(\rest_{\mod R}(\pi^{-1}\X))=\ssupp\X$; and
\item
$\pi(\thick_{\D^{\b}(R)}(\ipd^{-1}W))=\ssupp^{-1}W$.
\end{enumerate}

For (a), pick a prime ideal $\p\in\ipd(\rest_{\mod
R}(\pi^{-1}\X))=\ipd(\pi^{-1}\X\cap\mod R)$.
Then there is an $R$-module $M$ such that $X:=\pi M\in\X$ and
$\pd_{R_\p}M_\p=\infty$.
Hence, $X_\p\not\cong0$ in $\D_{\sg}(R_\p)$, which implies that $\p$ is in $\ssupp\X$.
Conversely, pick a prime ideal $\p\in\ssupp\X$.
Then there is an object $X\in\X$ such that $X_\p\not\cong0$ in $\D_{\sg}(R_\p)$.
Choose an object $Y\in\D^{\b}(R)$ with $\pi Y=X$.
There exists an exact triangle $P \to Y \to M[n] \rightsquigarrow$
in $\D^{\b}(R)$ such that $P\in\thick_{\D^{\b}(R)}R$, $M\in\mod R$, and $n\in\mathbb{Z}$.
Hence $X=\pi Y=(\pi M)[n]$, which implies that $\pi M$ is in $\X$, and
hence, $M$ is in $\pi^{-1}\X\cap\mod R=\rest_{\mod R}(\pi^{-1}\X)$.
Since $\p$ belongs to $\ipd(M)$, it belongs to $\ipd(\rest_{\mod
R}(\pi^{-1}\X))$.

For (b), let $A$ be an object in $\pi(\thick_{\D^{\b}(R)}(\ipd^{-1}W))$.
Then $A\cong \pi B$ for some $B\in\thick_{\D^{\b}(R)}(\ipd^{-1}W)$.
Hence, $\ssupp A=\ssupp(\pi B)=\ipd(B)$.
If $\p$ is a prime ideal outside $W$, then $(\ipd^{-1}(W))_\p$ is
contained in $\PD(R_\p)$, and so is
$(\thick_{\D^{\b}(R)}(\ipd^{-1}(W)))_\p$, which contains $B_\p$.
Hence, $B_\p$ has finite projective dimension as an $R_\p$-module,
which shows that $\ipd(B)$ is contained in $W$.
Therefore, $A$ belongs to $\ssupp^{-1}W$.
Conversely, let $A$ be an object in $\ssupp^{-1}W$, and write $A=\pi
B$ with $B\in\D^{\b}(R)$.
Then it is easy to see that $\ipd(B)$ is contained in $W$, and $B$
belongs to $\thick_{\D^{\b}(R)}(\ipd^{-1}W)$.
Hence, $A$ is in $\pi(\thick_{\D^{\b}(R)}(\ipd^{-1}W))$.
\qed
\end{para}

We close this section by presenting several examples of a Cohen-Macaulay non-Gorenstein local ring with quasi-decomposable maximal ideal; actually, it turns out that there are plenty of such examples.

\begin{ex}\label{mm}
Let $R$ be a $d$-dimensional Cohen-Macaulay local ring which is not a hypersurface.
Suppose that $k$ is infinite and that $R$ has minimal multiplicity.
Then $\m$ is quasi-decomposable.
In fact, we find a system of parameters $\xx$ of $R$ with $\m^2=\xx\m$.
Then we have $\m(\m/(\xx))=0$, which means that $\m/(\xx)$ is a $k$-vector space.
Since $R$ is not a hypersurface, the dimension of $\m/(\xx)$ as a $k$-vector space is at least two.
Hence $\m/(\xx)$ is decomposable.
\end{ex}

\begin{ex}\label{ex20170916}
Let $R$ be a 2-dimensional non-Gorenstein
normal local domain with a rational singularity. Then by~\cite[Theorem 3.1]{Huneke-Watanabe}, the ring $R$ has minimal multiplicity. Therefore, by Example~\ref{mm} it has a
quasi-decomposable maximal ideal.
\end{ex}

\begin{ex}\label{ex150217a}
Let $H=\langle pq+p+1, 2q+1, p+2\rangle$
be the numerical semigroup with $p,q>0$ and $\gcd(p + 2, 2q + 1) = 1$, and let $R=k[[H]]$ be (the completion of) the numerical semigroup ring of $H$ over a field $k$.
Then $R$ is a non-Gorenstein almost-Gorenstein (see \cite{almgor}) Cohen-Macaulay local domain of dimension $1$, and $\m$ is quasi-decomposable.
In fact, by virtue of \cite[Proposition 2(1)]{Mo}, we have
\begin{equation}\label{eq150217a}
R\cong \frac{k[[x,y,z]]}{\I\left(\begin{smallmatrix}
x&y&z\\
y^p&z^q&x
\end{smallmatrix}\right)}
\end{equation}
where $\I\,(-)$ stands for the ideal generated by $2\times2$-minors.
This implies $R/(z)\cong k[[x,y]]/(x^2,xy,y^{p+1})$, and hence we obtain a non-trivial direct sum decomposition $\m/(z)\cong x(R/(z))\oplus y(R/(z))$.
\end{ex}

\begin{ex}
Let $R=k[[H]]$ be a non-Gorenstein almost-Gorenstein numerical semigroup ring with embedding dimension $3$ and multiplicity $\le6$.
Then $\m$ is quasi-decomposable.
Indeed, by \cite[Proposition 2(2)]{Mo} we again have the isomorphism~\eqref{eq150217a} with $p,q>0$, and a similar argument as in Example~\ref{ex150217a} applies.
\end{ex}

\begin{ex}
For each of the Cohen-Macaulay local rings $R$ given in \cite[Examples 7.1, 7.2, 7.4 and 7.5]{crs} the maximal ideal is quasi-decomposable, since the quotient of a suitable system of parameters is isomorphic to the ring
$$
k[s,t]/(s^2,st,t^2)=k[s]/(s^2)\times_kk[t]/(t^2),
$$
where $s,t$ are indeterminates over the field $k$.
We should notice that these four examples also explain the independence of the assumptions of Theorem \ref{31} on the punctured spectrum.
(None of them has an isolated singularity.)
\end{ex}

\section{More on thick subcategories of $\mod R$}\label{b}

As another application of Theorem~\ref{1}, in this section we study the thick subcategories of $\mod R$ that contain a module of infinite projective dimension where $\m$ is quasi-decomposable.
We begin with making a lemma.

\begin{lem}\label{18}
Let $M$ be an $R$-module and $\xx=x_1,\dots,x_n$ be an $R$-sequence. Then $\xx$ is an $\syz_R^nM$-sequence.
\end{lem}

\begin{proof}
We use induction on $n$.
The assertion is trivial when $n=0$.
Let $n>0$.
The induction hypothesis implies that the sequence $\xx':=x_1,\dots,x_{n-1}$ is regular on both $\syz^{n-1}M$ and $\syz^nM=\syz^{n-1}(\syz M)$.
The exact sequence $$0\to\syz^nM\to F\to\syz^{n-1}M\to0$$ in which $F$ is free induces an exact sequence
$$
0 \to \syz^nM/\xx'\syz^nM\to F/\xx'F\to\syz^{n-1}M/\xx'\syz^{n-1}M\to0.
$$
This shows that the multiplication map of $\syz^nM/\xx'\syz^nM$ by $x_n$ is injective. Thus, $\xx=x_1,\dots,x_n$ is regular on $\syz^nM$.
\end{proof}

The next proposition plays a central role in this section.

\begin{prop}\label{3}
Suppose that $\m$ is quasi-decomposable with an $R$-sequence $\xx$ of length $n$. If $M$ is an $R$-module with $\pd_RM\ge n+2$, then $k$ is in $\X:=\thick_{\mod R}\{R,M\}$.
\end{prop}

\begin{proof}
We consider the following cases.

Case 1: $n=0$.
In this case $\m$ is decomposable, and we can apply Theorem \ref{1}.
Since $\X$ contains all the syzygies of $M$ and their direct summands, we are done.

Case 2: $n>0$.
It follows from Lemma \ref{18} that $\xx$ is an $\syz_R^nM$-sequence.
Using \cite[Lemma 1.3.5]{BH}, we have
$$
\pd_{R/(\xx)}\syz_R^nM/\xx \syz_R^nM=\pd_R\syz_R^nM\ge(n+2)-n=2.
$$
By Case 1 we have $k$ is in $\thick_{\mod R/(\xx)}\{R/(\xx),\syz_R^nM/\xx \syz_R^nM\}$. Therefore, $k$ is in $\thick_{\mod R}\{R/(\xx),\syz_R^nM/\xx \syz_R^nM\}$.
The modules $R/(\xx)$ and $\syz_R^nM/\xx \syz_R^nM$ are in $\thick_{\mod R}\{R,\syz_R^nM\}=\X$.
Therefore, $k$ belongs to $\X$, as desired.
\end{proof}

\begin{cor}\label{19}
Suppose that $R$ has an isolated singularity (that is, $\sing R$ is a point), and that $\m$ is quasi-decomposable with an $R$-sequence $\xx$ of length $n$.
\begin{enumerate}[\rm(1)]
\item
If $\X$ is a thick subcategory of $\mod R$ containing $R$ and at least one module of projective dimension $\ge n+2$, then $\X=\mod R$.
\item
There exists no thick subcategory $\X$ with $\PD(R)\subsetneq\X\subsetneq\mod R$.
\end{enumerate}
\end{cor}

\begin{proof}
(1) Since we assume that $R$ has an isolated singularity, by \cite[Theorem VI.8]{S} (see also \cite[Proposition 9]{KS}) we have $\thick_{\mod R}\{R,k\}=\mod R$.
The assertion now follows from Proposition \ref{3}.

(2) If $\X$ is a thick subcategory strictly containing $\PD(R)$, then $\X$ contains a module of infinite projective dimension, and $\X=\mod R$ by (1).
\end{proof}

\begin{disc}\label{30}
Corollary \ref{19}(2) holds whenever $R$ is a hypersurface with an isolated singularity.
In fact, since $R$ is a hypersurface, by \cite[Theorem 5.1(1)]{thd} there is a one-to-one correspondence between the thick subcategories of $\mod R$ containing $R$ and the specialization-closed subsets of $\sing R$.
As $R$ has an isolated singularity, $\sing R$ is trivial, and thus there is no thick subcategory $\X$ with $\PD(R)\subsetneq\X\subsetneq\mod R$.
\end{disc}

The following is also a corollary of Proposition \ref{3} (hence one of Theorem \ref{1}), which is a variant of \cite[Proposition 5.2]{crs}.

\begin{cor}
Let $R$ be a Cohen-Macaulay non-Gorenstein local ring with minimal multiplicity and infinite residue field.
Let $M$ be an $R$-module with infinite projective dimension.
Then $k$ belongs to $\thick_{\mod R}\{R,M\}$.
\end{cor}

\begin{proof}
In view of Example \ref{mm}, the maximal ideal of $R$ is quasi-decomposable. Hence the assertion follows from Proposition \ref{3}.
\end{proof}

\section{Vanishing of Ext and Tor}\label{c}

In this section we give some Ext and Tor vanishing results over local rings with quasi-decomposable maximal ideal as applications of Theorem~\ref{1}. Some of these results refine the existing results of~\cite{NSW} as we explain below.
We start by a lemma.

\begin{lem}\label{16}
Let $M$ and $N$ be $R$-modules.
Let $t\ge0$ be an integer.
\begin{enumerate}[\rm(1)]
\item
Suppose that $\m$ is a direct summand of $\syz^tM$.
\begin{enumerate}[\rm(a)]
\item
If $\Tor_{\ell}^R(M,N)=0$ for some $\ell\ge t+1$, then $\pd_RN<\infty$.
\item
If $\Ext^{\ell}_R(M,N)=0$ for some $\ell\ge t+\max\{1,\depth_R N\}$, then $\id_RN<\infty$.
\end{enumerate}
\item
Suppose that $\m$ is a direct summand of $\syz^tM\oplus\syz^{t+1}M$.
\begin{enumerate}[\rm(a)]
\item
If $\Tor_{\ell}^R(M,N)=\Tor_{\ell+1}^R(M,N)=0$ for some $\ell\ge t+1$, then $\pd_RN<\infty$.
\item
If $\Ext^{\ell}_R(M,N)=\Ext_R^{\ell+1}(M,N)=0$ for some $\ell\ge t+\max\{1,\depth_R N\}$, then $\id_RN<\infty$.
\end{enumerate}
\end{enumerate}
\end{lem}

\begin{proof}
(1a) Since $\ell-t>0$, we have $\Tor_{\ell}^R(M,N)\cong \Tor_{\ell-t}^R(\syz^tM,N)$.
By assumption we get $\Tor_{\ell-t+1}^R(k,N)\cong\Tor_{\ell-t}^R(\m,N)=0$.
Hence, $\pd_RN<\infty$.

(1b) As $\ell-t>0$, we have $\Ext^{\ell}_R(M,N)\cong \Ext^{\ell-t}_R(\syz^tM,N)$.
The assumption implies $\Ext_R^{\ell-t+1}(k,N)\cong\Ext_R^{\ell-t}(\m,N)=0$.
Since $\ell-t+1>\depth N$, by~\cite[Theorem (1.1)]{FFGR} we conclude that $\id_RN<\infty$.

(2a) Since $\ell-t>0$, we have
$$
0=\Tor_{\ell}^R(M,N)\oplus\Tor_{\ell+1}^R(M,N)\cong\Tor_{\ell-t}^R(\syz^tM\oplus\syz^{t+1}M,N),
$$
which contains $\Tor_{\ell-t}^R(\m,N)\cong\Tor_{\ell-t+1}^R(k,N)$ as a direct summand.
Hence $N$ has finite projective dimension.

(2b) This follows from a similar argument to the proofs of (2a) and (1b).
\end{proof}

Our Theorem \ref{1} recovers~\cite[Theorems 1.1]{NSW} as we state next.

\begin{cor}[Nasseh and Sather-Wagstaff]\label{15}
Assume that $R=S\times_kT$ is a fiber product, where $S$ and $T$ are local rings with common residue field $k$ and $S\ne k\ne T$.
Let $M$ and $N$ be $R$-modules.
\begin{enumerate}[\rm(1)]
\item
If $S$ or $T$ has depth $0$ and $\Tor_{\ell}^R(M,N)=0$ for some $\ell\ge5$, then $M$ or $N$ is $R$-free.
\item
If $\Tor_{\ell}^R(M,N)=\Tor_{\ell+1}^R(M,N)=0$ for some $\ell\ge5$, then $\pd_R M\le 1$ or $\pd_R N\le 1$.
\end{enumerate}
\end{cor}

\begin{proof}
(1) By Fact~\ref{fact300117a} the ring $R$ has depth $0$.
If $M$ is not $R$-free, then $\pd_R M=\infty$.
Theorem~\ref{1} implies that $\m$ is a direct summand of $\syz_R^3M$ or $\syz_R^4M$.
Lemma \ref{16}(1) now implies that $\pd_R N<\infty$, and $N$ is $R$-free.

(2) By Fact~\ref{fact300117a} the ring $R$ has depth $\le 1$.
Theorem~\ref{1} implies that $\m$ is a direct summand of $\syz_R^3M\oplus\syz_R^4M$ or $\syz_R^5M$.
Assume that $\pd_RM\ge2$. Note that $\ell\ge3+1$ and $\ell+1\ge5+1$. It is observed from (1a) and (2a) in Lemma \ref{16} that $\pd_R N<\infty$.
Hence, $\pd_RN\le 1$.
\end{proof}

An analogous argument to the proof of Corollary \ref{15} yields its Ext version.
Note that this highly refines the Ext vanishing results in~\cite{NSW}.

\begin{cor}\label{15'}
Assume that $R=S\times_kT$ is a fiber product, where $S$ and $T$ are local rings with common residue field $k$ and $S\ne k\ne T$.
Let $M$ and $N$ be $R$-modules.
\begin{enumerate}[\rm(1)]
\item
If $S$ or $T$ has depth $0$ and $\Ext^{\ell}_R(M,N)=0$ for some $\ell\ge4+\max\{1,\depth_R N\}$, then $M$ is $R$-free or $N$ is $R$-injective.
\item
If $\Ext^{\ell}_R(M,N)=\Ext^{\ell+1}_R(M,N)=0$ for some $\ell\ge4+\max\{1,\depth_R N\}$, then $\pd_R M\le 1$ or $\id_R N\le 1$.
\end{enumerate}
\end{cor}

To get our next applications, we prepare a lemma.

\begin{lem}\label{35}
Let $\xx=x_1,\dots,x_n$ be an $R$-sequence such that $\m/(\xx)$ is decomposable.
Then $n$ is equal to either $\depth R-1$ or $\depth R$.
\end{lem}

\begin{proof}
By Fact~\ref{fact300117a} we have $\depth R/(\xx)\le1$. This shows the assertion.
\end{proof}

The following results are generalizations of Corollaries \ref{15} and \ref{15'} to a local ring with quasi-decomposable maximal ideal.

\begin{cor}\label{tcm}
Let $R$ be a local ring with $\depth R=d$ such that $\m$ is
quasi-decomposable with an $R$-sequence $\xx=x_1,\dots,x_n$.
Let $M$ and $N$ be $R$-modules for which there exists an integer $t\ge\max\{5,n+1\}$ such that $\Tor^R_i(M,N)=0$ for all $t+n\le i\le t+n+d$.
Then $\pd_RM<\infty$ or $\pd_RN<\infty$.
\end{cor}

\begin{proof}
Note that $t-n>0$. By Lemma~\ref{18} we see that $\xx$ is a regular sequence on both $X:=\syz^nM$ and $Y:=\syz^nN$. We also have $\Tor_i^R(X,Y)=0$ for all $t-n\le i\le t-n+d$.
Using the long exact sequences arising from the short exact sequences
$$
0 \to Y/(x_1,\dots,x_{j-1})Y \xrightarrow{x_j} Y/(x_1,\dots,x_{j-1})Y \to  Y/(x_1,\dots,x_j)Y \to 0
$$
for $1\le j\le n$, we observe that $\Tor_i^R(X,Y/\xx Y)=0$ for all $t\le i\le t-n+d$.
Applying \cite[\S18, Lemma 2(3)]{Mm}, we obtain $\Tor_i^{R/(\xx)}(X/\xx X,Y/\xx Y)=0$ for all $t\le i\le t-n+d$.

According to Lemma \ref{35}, the integer $n$ is either $d$ or $d-1$.
If $n=d$, then $\Tor_t^{R/(\xx)}(X/\xx X,Y/\xx Y)=0$ and $R/(\xx)$ has depth zero.
Fact \ref{fact300117a} and Corollary \ref{15}(1) imply that either $X/\xx X$ or $Y/\xx Y$ is $R/(\xx)$-free.
If $n=d-1$, then we have $\Tor_t^{R/(\xx)}(X/\xx X,Y/\xx Y)=\Tor_{t+1}^{R/(\xx)}(X/\xx X,Y/\xx Y)=0$.
Fact \ref{fact300117a} and Corollary \ref{15}(2) imply that $\pd_{R/(\xx)}(X/\xx X)\le 1$ or $\pd_{R/(\xx)}(Y/\xx Y)\le 1$.
Hence, $\pd_RX<\infty$ or $\pd_RY<\infty$ by \cite[Lemma 1.3.5]{BH}. Therefore, $\pd_RM<\infty$ or $\pd_RN<\infty$.
\end{proof}

\begin{cor}\label{ecm}
Let $R$ be a $d$-dimensional Cohen-Macaulay local ring such that $\m$ is
quasi-decomposable with an $R$-sequence $\xx=x_1,\dots,x_n$.
Let $M$ and $N$ be $R$-modules, and set $s:=d-\depth M\,(\ge0)$.
Suppose that there exists an integer $t\ge5$ such that $\Ext_R^i(M,N)=0$ for all $t+s\le i\le t+s+d$.
Then $\pd_RM<\infty$ or $\id_RN<\infty$.
\end{cor}

\begin{proof}
Without loss of generality we may assume that $R$ is complete, so that $R$ admits a canonical module.
Let $L:=\syz_R^sM$ and note then that $L$ is maximal Cohen-Macaulay.
We now have $\Ext_R^i(L,N)=0$ for all $t\le i\le t+d$.
By virtue of \cite[Theorem 1.1]{AB}, we get an exact sequence $0 \to Y \to X \to N \to 0$
in which $X$ is maximal Cohen-Macaulay and $Y$ has finite injective dimension.
Since $\Ext_R^m(L,Y)=0$ for all $m>0$, we get $\Ext_R^i(L,X)=0$ for all $t\le i\le t+d$.
Using the long exact sequences arising from the short exact sequences
$$
0 \to L/(x_1,\dots,x_{j-1})L \xrightarrow{x_j} L/(x_1,\dots,x_{j-1})L \to  L/(x_1,\dots,x_j)L \to 0
$$
for $1\le j\le n$, we see that $\Ext_R^i(L/\xx L,X)=0$ for all $t+n\le i\le t+d$.
It follows from \cite[\S18, Lemma 2(i)]{Mm} that $\Ext_{R/(\xx)}^i(L/\xx L,X/\xx X)=0$ for all $t\le i\le t+d-n$.

Lemma \ref{35} says that $n=d$ or $n=d-1$, and note that
$$
5=4+1=4+\max\{1,\depth_{R/(\xx)}X/\xx X\}.
$$
When $n=d$, the ring $R/(\xx)$ has depth zero and we have $\Ext_{R/(\xx)}^t(L/\xx L,X/\xx X)=0$.
Combining Fact \ref{fact300117a} and Corollary \ref{15'}(1), we conclude that $L/\xx L$ is $R/(\xx)$-free or $X/\xx X$ is $R/(\xx)$-injective.
When $n=d-1$, we have $\Ext_{R/(\xx)}^t(L/\xx L,X/\xx X)=\Ext_{R/(\xx)}^{t+1}(L/\xx L,X/\xx X)=0$.
In this case, Fact \ref{fact300117a} and Corollary \ref{15'}(2) imply that $\pd_{R/(\xx)}(L/\xx L)\le1$ or $\id_{R/(\xx)}(X/\xx X)\le1$.
Therefore, $\pd_RL<\infty$ or $\id_RX<\infty$; see \cite[Lemma 1.3.5 and Corollary 3.1.15]{BH}. Thus, $\pd_RM<\infty$ or $\id_RN<\infty$.
\end{proof}

\begin{disc}\label{disc04042018a}
We do not know how to get a complete generalization of Corollary~\ref{15'} to the
higher-dimensional case (with no Cohen-Macaulay assumption) like what we did in Corollary~\ref{tcm} as a generalization of Corollary~\ref{15}.
However, if we assume infinitely many Ext vanishings, then we can prove the following result using Corollary~\ref{tcm} and~\cite[Proposition 6.5]{AINSW} (see also the proof of~\cite[Proposition 4.1]{NSW}):
\end{disc}

\begin{cor}\label{ecm1}
Let $R$ be a local ring with quasi-decomposable maximal ideal $\m$.
Let $M$ and $N$ be $R$-modules such that $\Ext_R^i(M,N)=0$ for all $i\gg 0$.
Then $\pd_RM<\infty$ or $\id_RN<\infty$.
\end{cor}

\section*{Acknowledgments}
We thank the referee for reading the paper and for helpful suggestions.

%\bibliography{../+new}

\providecommand{\bysame}{\leavevmode\hbox to3em{\hrulefill}\thinspace}
\providecommand{\MR}{\relax\ifhmode\unskip\space\fi MR }
\providecommand{\MRhref}[2]{%
  \href{http://www.ams.org/mathscinet-getitem?mr=#1}{#2}
}
\providecommand{\href}[2]{#2}

\end{document}